\documentclass{amsart}

\usepackage{amsmath,amssymb,amsfonts,amstext,amsthm,bm}
\usepackage{url}
\usepackage{graphicx}
\usepackage{hyperref}
\usepackage{tikz}
\usepackage[all,cmtip]{xy} 
\usepackage{multicol}
\usepackage{verbatim}
\input xy
\xyoption{all}
\usepackage[all]{xy}
\usepackage{xcolor}
\usepackage{enumitem}

\def\Z{{\mathbb Z}}

\newtheorem{coro}{Corollary}[section]

\newtheorem{lemma}[coro]{Lemma}

\newtheorem{prop}[coro]{Proposition}

\newtheorem{rem}[coro]{Remark}

\newtheorem{thm}[coro]{Theorem}

\let \ep=\epsilon
\let \ga=\gamma

\let \be=\beta

\let \al=\alpha
\let \pr=\prime

\let \mb=\mathbb

\let \wt=\widetilde

\let \ra=\rightarrow

\usepackage{faktor}

\newtheorem{defn}[coro]{Definition}

\def\Ap{{\operatorname{Ap}}}

\title{Cusps in $\mb{C}^3$ with prescribed ramification}
\date{\today}
\author{Ethan Cotterill}

\address{Instituto de Matem\'atica, Estat\'istica, e Computa\c{c}\~ao Cient\'ifica, UNICAMP, Rua S\'ergio Buarque de Holanda, 651, 13.083-859 Campinas SP, Brazil}

\email{cotterill.ethan@gmail.com}

\author{Nathan Kaplan}

\address{Department of Mathematics, University of California, Irvine, 419 Rowland Hall, Irvine, CA 92697, USA}

\email{nckaplan@math.uci.edu}

\author{Renata Vieira Costa}

\address{Instituto de Matem\'atica, UFF, Rua Prof Waldemar de Freitas, S/N, 24.210-201 Niter\'oi RJ, Brazil
}

\email{renatavico8@gmail.com}

\begin{document}

\maketitle 

\begin{abstract}
We study value semigroups associated to germs of maps $\mb{C} \ra \mb{C}^3$ with fixed ramification profiles in a distinguished point. We then apply our analysis to deduce that Severi varieties of unicuspidal rational fixed-degree curves with value semigroup ${\rm S}$ in $\mb{P}^3$ are often reducible when ${\rm S}$ is either 1) the semigroup of a generic cusp whose ramification profile 
 is either a triple of successive multiples of a fixed integer or is a {\it supersymmetric} triple; or 2) a supersymmetric semigroup with ramification profile given by a supersymmetric triple. In doing so, we uncover 
new connections with additive combinatorics and number theory. 
\end{abstract}

\section{Introduction}
{\it Cusps}, or unirational algebraic curve singularities, are interesting for many reasons. In \cite{CLM}, we showed that the {\it generalized Severi variety} $M^n_{d,g}$ of morphisms $\mb{P}^1 \ra \mb{P}^n$ of degree $d$ with images of arithmetic genus $g$ is often not irreducible whenever $n \geq 3$. By contrast, the (classical) Severi variety $V_{d,g}$ of algebraic {\it plane} curves of fixed degree and genus is always irreducible, and equal to the closure of the subvariety of plane curves with the maximal number of simple nodes \cite{H}. In particular, {\it rational} plane curves of fixed degree $d$ and arithmetic genus $g$ always belong to the closure of the $g$-nodal locus in degree $d$. When $n \geq 3$, on the other hand, $M^n_{d,g}$ often contains components of dimension strictly larger than that of the $g$-nodal locus. Indeed, in \cite{CLM} we explicitly constructed such components by producing rational curves with cusps whose desingularizations had fixed {\it ramification profiles} equal to sequences of consecutive even numbers, but were otherwise generic. We were then able to conclude by bounding the genus (i.e., delta-invariants) of these {\it generic cusps} from below.

\medskip
In this paper, we will shed some additional light on the structure of $M^3_{d,g}$. Namely, we will determine the genera and dimensions of spaces of unicuspidal rational curves in $\mb{P}^3$ of fixed (sufficiently large) degree whose cusps have fixed ramification profiles ${\bf r}=(\al,\be,\ga)$ in a (countable) number of interesting cases. More precisely, we study the {\it generic value semigroup} ${\rm S}={\rm S}({\bf r})$ associated with a ramification profile ${\bf r}$ when ${\bf r}$ is either a triple of consecutive multiples of a given integer, or a {\it supersymmetric} triple in the sense of \cite{FGH}. As a byproduct of our analysis we show that asymptotically, the associated mapping spaces $M^3_{d,g; {\rm S}, {\bf r}}$ are always of larger-than-expected dimension. We also show that supersymmetric sequences ${\bf r}$ generate semigroups ${\rm S}$ whose associated mapping spaces $M^3_{d,g; {\rm S}, {\bf r}}$ are of larger-than-expected dimension; and in the process, we uncover an interesting connection with lattice point counts for simplices.

\subsection{Diophantine equations, Betti elements, and Severi varieties with excess dimension}
Given postive integers $\al < \be< \ga$, the value semigroup ${\rm S}$ of a generic map of power series $f=(f_1,f_2,f_3): \mb{C} \ra \mb{C}^3$ with vanishing orders ${\bf r}=(\al,\be,\ga)$ in a fixed point trivially contains $\al$, $\be$, and $\ga$, but determining the remaining generators of ${\rm S}$ is more delicate. The basic point here is that whenever there exist nonnegative integers $x$, $y$, $z$ and a positive integer $N$ for which either
\begin{equation}\label{diophantine}
   N =\al x+ \be y= \ga z,
   N =\al x+ \ga y= \be z, \text{ or }
   N = \be y+ \ga z = \al x.
\end{equation}
the element $N+1$ then belongs to ${\rm S}$. Indeed, whenever one of the three equalities in \eqref{diophantine} holds, $N+1$ is the $t$-valuation of the difference of monomials in the parameterizing functions $f_i$. As a result, coincidences among {\it factorizations} of integers as positive integer linear combinations of a fixed list of generators play a critical role in determining the value semigroup of a generic cusp associated with a fixed ramification profile. 

\medskip
 A closely related problem, studied in \cite{CLM, CLMR}, is that of determining the number of algebraically independent conditions imposed by cusps of a particular value semigroup and ramification type $({\rm S},{\bf r})$ on parameterizations $\mb{P}^1 \ra \mb{P}^n$ of fixed degree $d$. The structure of the factorization space of a given numerical semigroup ${\rm S}$, in turn, is shaped by its set of {\it Betti elements}, i.e., by those elements whose associated factorization graphs are disconnected; see \cite{Ch}. 

\medskip
Numerical semigroups whose sets of minimal generators contain subsets in arithmetic progression generally contain many Betti elements. On the other hand, numerical semigroups whose sets of Betti elements are singletons are precisely those minimally generated by {\it supersymmetric} $n$-tuples $\frac{a_1 \dots a_n}{a_i}$, $i=1,\dots,n$ derived from pairwise relatively prime integers $a_1,\dots,a_n$ \cite[Example 12]{GSOR}. In this paper, we focus on two objects of study:
\begin{enumerate}
\item 
Severi varieties of 
rational curves with generic cusps whose ramification profiles are triples that are either in arithmetic progression or supersymmetric; and
\item Severi varieties
of rational curves with supersymmetric value semigroup 
and ramification profiles that are supersymmetric triples ${\bf r}$. 
\end{enumerate}

\medskip
We say that a Severi variety $M^n_{d,g; {\rm S},{\bf r}}$ is {\it excess-dimensional} whenever its codimension inside $M^n_d$ is strictly less than $(n-2)g$, the codimension of $g$-nodal rational curves of degree $d$ in $\mb{P}^n$ when $d$ is sufficiently large. Given a strictly-increasing sequence of (strictly-) increasing positive integers ${\bf r}=(r_1,\dots,r_n)$, a {\it generic cusp with ramification profile ${\bf r}$} is an $n$-tuple of power series $f_i(t)$, $i=1,\dots,n$ whose $t$-adic valuations comprise ${\bf r}$ and whose higher-order coefficients are algebraically independent. A {\it generic semigroup adapted to ${\bf r}$} is the value semigroup associated with such an $n$-tuple of power series $f_i$, i.e., the set of valuations realized by the algebra they generate. 

\medskip
The characteristic feature of a generic semigroup ${\rm S}$ adapted to a given ramification profile ${\bf r}=(r_1,\dots,r_n)$ is that its associated codimension, i.e., that of $M^n_{d,g; {\rm S},{\bf r}}$, is predictable: indeed, it is precisely 
$\sum_{i=1}^n (r_i-i)-1$, in which the instance of $-1$ accounts for variation in the preimage of the underlying cusp. To certify that a Severi variety $M^n_{d,g; {\rm S},{\bf r}}$ is excess-dimensional when ${\rm S}$ is the generic semigroup adapted to ${\bf r}$, it therefore suffices to bound the genus $g({\rm S})$ of ${\rm S}$ from below; to do so, we use elementary valuation-theoretic arguments. Genericity of (the higher-order coefficients of) the power series $f_i$ guarantees that when their respective valuations $m=m_i$ are clustered together, the large-$m$ asymptotics of $g({\rm S})$ are robust. Precisely determining the generic semigroup ${\rm S}$ adapted to a given ramification profile ${\bf r}$, appears to be difficult, if in principle algorithmically resolvable; see Remark~\ref{discrete_dynamical_system}. 

\medskip
It is also natural to try producing excess-dimensional cusps of a fixed type $({\rm S}, {\bf r})$, while fixing ${\rm S}$ but no longer requiring it to be the generic value semigroup adapted to ${\bf r}$ as in the preceding paragraph. One particularly interesting situation is that in which ${\bf r}$ comprises a minimal generating set for ${\rm S}$. In this paper, we further specialize to the case of supersymmetric semigroups. These are distinguished by their singleton Betti sets; moreover, according to \cite{CLMR}, their associated codimensions are combinatorially predictable, and modeled on the discrete volumes of explicit lattice simplices.  We exploit the symmetries of supersymmetric semigroups, in tandem with known estimates for lattice point counts in rational simplices, to deduce that the Severi varieties $M^3_{d,g;{\rm S},{\bf r}}$ derived from supersymmetric triples ${\bf r}$ and the semigroups ${\rm S}$ they generate are very often excess-dimensional. 

\subsection{Conventions} In this paper, we work over $\mb{C}$, and we make use of several standard semigroup- and singularity-theoretic notions. A {\it cusp} is a unibranch curve singularity; as such, it is locally prescribed by a ring map $\phi: \mb{C}[x_1,\dots,x_n] \ra \mb{C}[\![t]\!]$ whose image is its {\it local algebra}. The ring $\mb{C}[\![t]\!]$ is equipped with the {\it $t$-adic valuation} $v_t$ that computes the order of vanishing of a power series in $t$ at $t=0$; and the image of $v_t \circ \phi$ is the {\it value semigroup} of the cusp. 

\medskip
Every value semigroup $\rm{S}$ is a {\it numerical semigroup}, i.e., it has finite complement in $\mb{N} =\{0,1,2,\ldots\}$. 
The cardinality (resp., largest element) of $\mb{N} \setminus {\rm S}$ is the {\it genus $g=g({\rm S})$} (resp., {\it Frobenius number} $F=F({\rm S})$ of ${\rm S}$). The {\it conductor} $c$ of ${\rm S}$ is $F+1$. A numerical semigroup ${\rm S}$ is {\it symmetric} if and only if for every $x \in \mb{N}$, exactly one of $x$ or $F-x$ belongs to ${\rm S}$.
Positive integers $n_1,\ldots, n_\ell$ comprise a set of {\it generators} for ${\rm S}$ if every $s\in {\rm S}$ may be realized as an $\mb{N}$-linear combination of $n_1,\ldots, n_{\ell}$; whenever this is the case, we write ${\rm S} = \langle n_1,\ldots, n_{\ell}\rangle$. Sets of generators for a given semigroup ${\rm S}$ are partially-ordered by inclusion; and there is a unique minimal generating set, or set of {\it minimal generators}. We will often reference factorizations with respect to generating sets fixed at the outset, which may or may not be minimal. 
A {\it factorization} of $s \in {\rm S}$ with respect to generators $n_1,\dots,n_{\ell}$ is an $\ell$-tuple $(a_1,\dots,a_{\ell}) \in \mb{N}^{\ell}$ such that $\sum_{i=1}^{\ell} a_i n_i=s$; and $Z(s)$ denotes the set of all factorizations of $s$.

\subsection{Roadmap}
The plan for the remainder of the paper is as follows. In Section~\ref{generic_arithmetic_semigroups}, we produce explicit lower and upper bounds on the genera of generic cusps with fixed ramification profiles in arithmetic progression. Our approach to lower bounds on the genera of generic semigroups is predicated on Lemma~\ref{valuation_bound}, which gives an upper bound on the $t$-adic valuation of a linear combination of power series $f_i(t)$ in $t$ with generic higher-order coefficients. We then use Lemma~\ref{valuation_bound} to produce explicit upper and lower bounds on the $t$-adic valuations of polynomials of each degree $k \geq 1$ in triples of power series $f_i(t)$, $i=1,2,3$ whose underlying $t$-adic valuations are fixed. Lemma~\ref{valuation_bound_bis}, and its corollaries~\ref{cor_val1}, \ref{cor_val2} and \ref{genus_lb_generic_semigroup} are results tailored specifically to the case of ramification triples in arithmetic progression, which we generalize in Proposition~\ref{general_lower_genus_bound} to the setting of arbitrary ramification triples. The upshot is that whenever $m$ is sufficiently large and $a$ and $b$ are sufficiently small relative to $m$, the Severi variety of unicuspidal rational space curves with generic cusps of ramification type $(m,m+a,m+b)$ is excess-dimensional.

\medskip
We then turn to approximations to the generic semigroup ${\rm S}$ associated with a triple of consecutive even numbers. 
Theorem~\ref{generic_semigroup_consecutive_multiples} 
describes a first approximation
when ${\bf r}$ is a triple of consecutive multiples of an integer $m\geq 2$. Characterizing its gaps involves computing the Ap\'ery set of minimal elements in each congruence class modulo the multiplicity of the approximating semigroup; see Proposition~\ref{ApSprime}. This leads to Theorem~\ref{generic_semigroup_consecutive_multiples}, which gives upper bounds on the genera of generic semigroups derived from triples of consecutive multiples of a fixed integer. 

\medskip
Finally, Section~\ref{semigroups_from_supersymmetric_triples} is a close study of two distinguished classes of three-dimensional cusps with supersymmetric ramification profiles $(ab,ac,bc)$; namely, those with value semigroups ${\rm S}$ {\it generated} by supersymmetric ramification triples, and those whose parameterizations have generic higher-order coefficients. It is well-known that supersymmetric semigroups are {\it symmetric}, and when used in tandem with a dimension theorem for supersymmetric cusps proved in \cite{CLMR}, this symmetry allows us to rewrite the codimension associated with the pair $({\rm S},(ab,ac,bc))$ as a quantity with asymptotics modeled on the discrete volume of the lattice simplex with vertices $(0,0,0), (a,0,0), (0,b,0), (0,0,c)$.  In Theorem~\ref{excess_thm}, we show that a large class of cusps with supersymmetric value semigroups generated by supersymmetric ramification triples are excess-dimensional; and in Theorem~\ref{excess_thm2} we prove an analogous result for generic semigroups adapted to supersymmetric triples.
In Propositions~\ref{genus_S'}
and \ref{frob_S'}, we  compute the 
genus and Frobenius number, respectively, of a semigroup ${\rm S}^{\pr}$ containing ${\rm S}$ that approximates (and is contained in) the generic semigroup adapted to $(ab,ac,bc)$.

\section{Generic value semigroups with ramification $(m\ell, m\ell+m, m\ell+2m)$}\label{generic_arithmetic_semigroups}

As explained in \cite{CLM}, conditions imposed on points of the mapping space $M^n_d$ by cusps with numerical value semigroup ${\rm S}$ and ramification profile ${\bf r}$ are of two basic types. {\it Ramification} conditions, which may be read directly from ${\bf r}$, are linear in the coefficients of holomorphic maps, once the preimage $P$ of the cusp is fixed; there are $r_P= \sum_{i=1}^n (r_i-i)$ of these. Conditions {\it beyond ramification}, on the other hand, reflect (finer properties of) the additive structure of ${\rm S}$. Letting $b_P$ denote the number of algebraically independent conditions beyond ramification, the total number of conditions imposed by cusps of type $({\rm S},{\bf r})$ is $r_P+b_P-1$ whenever $d \geq \max(2g-2,n)$.

\medskip
For a {\it generic} cusp, we have $b_P=0$, and consequently the corresponding Severi variety $M^n_{d; ({\rm S},{\bf r})}$ is of codimension $\sum_{i=1}^n (r_i-i) -1$; letting $g=g({\rm S})$ denote the delta-invariant of the generic cusp, it therefore suffices to show that $\sum_{i=1}^n (r_i-i) \leq (n-2)g$ in order to conclude that the Severi variety $M^n_{d,g}$ is reducible. In particular, when $n=3$, reducibility of $M^3_{d,g}$ is guaranteed whenever $g \geq r_1+r_2+r_3-6$ and $d$ is sufficiently large relative to $g$.

\medskip
In this section we compute explicit lower and upper bounds on the genus of the value semigroup ${\rm S}$ of a generic cusp with ramification profile equal to a triple ${\bf r}=(m\ell,m\ell+m,m\ell+2m)$ of consecutive multiples of a fixed integer $m$. Our lower bounds imply that $g({\rm S}) \geq 3m\ell+3m-6$ whenever $\ell$ is sufficiently large, and consequently
$M^3_{d,g}$ is reducible. Our upper bounds on $g({\rm S})$, on the other hand, are realized as the genera of semigroup {\it approximations} contained in ${\rm S}$. As we explain in Remark~\ref{discrete_dynamical_system} below, the algorithmic approximation procedure we use may be viewed as a discrete dynamical system, and is therefore of independent interest.

\subsection{Lower bounds on the genera of generic semigroups}\label{lower_bounds_on_genus}

Our approach is predicated on repeated application of generalizations of the following elementary lemma.  Given a power series $f(t) = \sum_{n= 0}^\infty a_n t^n$ and a nonnegative integer $j$, $[t^j]f$ denotes the coefficient of $t^j$ in $f$ and $v_t(f)$ denotes the $t$-adic valuation of $f$.

\begin{lemma}\label{valuation_bound}
Suppose that $g_1,\ldots, g_N$ are generic power series in $t$.\footnote{Here ``genericity" means having generic higher-order coefficients, once the lowest-order terms are fixed.}  For any choice of nonzero coefficients $\al_1,\ldots, \al_N \in \mb{C}$, we have 
\[
v_t\left(\sum_{i=1}^N \al_i g_i\right) \le \min_i v_t(g_i) + N-1.
\]
\end{lemma}

\begin{proof}
Suppose first that $v_t(g_i)=\mu$ for every $i=1,\dots,N$; say $g_i(t)= \sum_{n=0}^{\infty} a_{i,\mu+n} t^n$ with $a_{i,\mu} \neq 0$ for every $i=1,\dots,N$. For every integer $\ga>0$, we have
\begin{equation}\label{linear_sys_eq}
\begin{split}
v_t\bigg(\sum_{i=1}^N \al_i g_i\bigg) \geq \mu+\ga &\iff \sum_{i=1}^N \al_i [t^{\mu+j}]g_i=0 \text{ for every }j=0,\dots,\ga-1 \\
&\iff \sum_{i=1}^N \al_i a_{i,\mu+j}=0 \text{ for every }j=0,\dots,\ga-1.
\end{split}
\end{equation}
Taking $\ga=N$, this means that
\begin{equation}\label{lin_sys_eq_matrix_version}
\left(
\begin{array}{cccccc}
a_{1,\mu} & \cdots & \cdots & a_{N,\mu}\\
a_{1,\mu+1} & \cdots & \cdots & a_{N,\mu+1}\\
\cdots & \cdots & \cdots & \cdots \\
a_{1,\mu+N} & \cdots & \cdots & a_{N,\mu+N}
\end{array}
\right)
\left(
\begin{array}{cccccc}
\al_1 \\
\al_2 \\
\cdots \\
\cdots \\
\al_N
\end{array}
\right)
=
\left(
\begin{array}{cccccc}
0 \\
0 \\
\cdots \\
\cdots \\
0
\end{array}
\right).
\end{equation}
Genericity of the nonzero coefficients $a_{i,k}$ means that the $N \times N$ coefficient matrix on the left-hand side of the matrix equation \eqref{lin_sys_eq_matrix_version} is invertible, which in turn forces the scalar multiplier vector $[\al_1,\cdots,\al_N]^T$ to be zero, a contradiction. It follows immediately that the system \eqref{linear_sys_eq} is insoluble whenever $\ga \geq N$, and consequently $v_t(\sum_{i=1}^N \al_i g_i) \leq \mu+N-1$ as desired.

\medskip
\noindent Now assume more generally that
\[
\begin{split}
    v_t(g_1)&=\dots=v_t(g_{N_1})= \mu_1 \\    v_t(g_{N_1+1})&=\dots=v_t(g_{N_2})= \mu_2 \\
    &\cdots \\    v_t(g_{N_{\ga-1}+1})&=\dots=v_t(g_{N_{\ga}})= \mu_{\ga}
\end{split}
\]
in which $\mu_1<\mu_2<\cdots<\mu_{\ga}$ and $N_1<N_2<\cdots<N_{\ga}$ are positive integers with $N_{\ga}=N$. For each $i=1,\ldots, \ga$, let $L_i$ denote a linear combination of 
$g_{N_{i-1}+1}, \dots, g_{N_i}$, where we set $N_0=0$. Using our analysis of the preceding paragraph, we have $v_t(L_1) \leq \mu_1+N_1-1$ and $v_t(L_i) \leq \mu_i+N_i-N_{i-1}$ for every $i=2,\dots,\ga$. Moreover, in order for the maximum possible valuation of $\sum_{i=1}^{\ga} L_i$ to be achieved, we must have $\mu_2=\mu_1+N_1-1$ and $\mu_i= \mu_{i-1}+N_{i-1}-N_{i-2}$ for every $i=3,\dots,\ga$. 
We now claim that
$v_t(\sum_{i=1}^N L_i) \leq \mu_1+N-1$. Indeed, say $L_1= \sum_{j=1}^{N_1} \al_j g_j$, $L_i=\sum_{j=N_{i-1}+1}^{N_i} \al_j g_j$ for every $j=2,\dots,\ga$, and that $\mu_2=\mu_1+N_1-1$ and $\mu_i= \mu_{i-1}+N_{i-1}-N_{i-2}$ for every $i=3,\dots,\ga$. We then have 
$v_t(\sum_{i=1}^\ga L_i) \geq \mu_1+N$ if and only if 
\begin{equation}\label{lin_sys_eq_general}
\begin{split}
\sum_{i=1}^{N_1} \al_i a_{i,\mu_1+j}&=0 \text{ for }j=0,\dots,N_1-1 \\
\sum_{i=1}^{N_2} \al_i a_{i,\mu_2+j}&=0 \text{ for }j=0,\dots,N_2-N_1-1 \\
&\cdots\\
\sum_{i=1}^{N_{\ga}} \al_i a_{i,\mu_{\ga}+j}&=0 \text{ for }j=0,\dots,N_{\ga}-N_{\ga-1}-1.
\end{split}
\end{equation}
The linear system of $N$ equations \eqref{lin_sys_eq_general} is a concatenation of the systems \eqref{linear_sys_eq}, and it is equivalent to a generalized version of the matrix equation \eqref{lin_sys_eq_matrix_version} whose associated coefficient matrix is a concatenation of the coefficient matrices in \eqref{lin_sys_eq_matrix_version}.
Genericity of the nonzero coefficients $a_{i,k}$ ensures that the coefficient matrix is invertible; it follows that the linear system of $N$ equations \eqref{lin_sys_eq_general} is insoluble, and therefore that $v_t(\sum_{i=1}^\ga L_i) \leq \mu_1+N-1$ as desired. 
\end{proof}

\noindent Given a polynomial $G(x_1,\ldots,x_n)$, we let $G_k$ denote its $k$\textsuperscript{th} graded piece. 
\begin{lemma}\label{valuation_bound_bis}
Let $\ell \geq 1$ be a positive integer. Suppose $f_1, f_2, f_3$ are generic power series in $t$ with $v_t(f_1) = 2\ell, v_t(f_2) = 2\ell+2,$ and $v_t(f_3) = 2\ell+4$.  Suppose $G(f_1,f_2,f_3)$ is a nonzero polynomial, and let $k$ be the minimum nonnegative integer for which the associated homogeneous graded piece $G_k$ is nonzero. Then, assuming that $4k+ \binom{k+2}{2}<2\ell$, we have
\[
k(2\ell) \leq v_t(G(f_1,f_2,f_3)) \le k(2\ell+4) + \binom{k+2}{2}-1.
\]
\end{lemma}

\begin{proof}
Write $f_i(t)= \sum_{n=0}^{\infty} a_{i,2\ell+2i-2+n} t^{2\ell+2i-2+n}$, $i=1,2,3$. The fact that $v_t(G(f_1,f_2,f_3)) \geq k(2\ell)$ follows immediately from the subadditivity of $v_t$, together with the fact that $f_1^k$ is the degree-$k$ monomial in $f_1$, $f_2$, and $f_3$ of minimal valuation $v_t=k(2\ell)$. To prove the upper bound, note that there are $\binom{k+2}{2}$ monomials of degree $k$ in $f_1$, $f_2$, and $f_3$, each of which has valuation at most $k(2\ell+4)$; so to show that $v_t(G(f_1,f_2,f_3)) \le k(2\ell+4) + \binom{k+2}{2}-1$
it suffices to show that the valuation of the sum of any $N$ monomials in the expansion of $G_k$ is at most $N-1$ more than the
minimum of their valuations. Assume first that the monomials $m_i=m_i(f_1,f_2,f_3)$, $i=1,\dots,N$ in question have equal valuations $v_t=\mu$. 
We will show that for every choice of $(\al_1,\dots,\al_N) \in \mb{C}^N$ with $\al_i \neq 0$ for every $i=1,\dots,N$, we have $v_t(\sum_{i=1}^N \al_i m_i) \leq \mu+N-1$. Indeed, for every integer $\ga>0$, we have
\begin{equation}\label{sys_eq_vt}
v_t\bigg(\sum_{i=1}^N \al_i m_i\bigg) \geq \mu+\ga \iff \sum_{i=1}^N \al_i [t^{\mu+j}]m_i=0, \text{ for every }j=0,\dots,\ga-1.
\end{equation}
When $\ga=N$, the
$N \times N$ matrix $M_N$ with $(i,j)$-th entry $[t^{\mu+j}]m_i$, where $i=1,\dots,N$ and $j=0,\dots,N-1$, has nonzero determinant. Indeed, for every $i=1,\dots,N$, we have $m_i=t^{\mu} \cdot g_{i,1}^{k_{i,1}} \cdot g_{i,2}^{k_{i,2}} \cdot g_{i,3}^{k_{i,3}}$ for power series $g_{i,j}=\sum_{\ell=0}^{\infty} b^{i,j}_{\ell} t^{\ell}$, $j=1,2,3$ with generic coefficients $b^{i,j}_{\ell}$ in {\it every} nonnegative degree $\ell$, and suitable natural number exponents $k_{i,1}$, $k_{i,2}$, $k_{i,3}$. Via the chain rule and multilinearity of the determinant, it is easy to see that $\det(M_N)$ includes a non-cancellable contribution from the determinant of
\begin{equation}\label{generic_matrix_bis}
\wt{M_N}=\left(
\begin{array}{cccccc}
\sum_{j=1}^3 k_{1,j} \wt{b}^{1,j}_0 & \cdots & \sum_{j=1}^3 k_{N,j} \wt{b}^{N,j}_0\\
\sum_{j=1}^3 k_{1,j} \wt{b}^{1,j}_1 & \cdots & \sum_{j=1}^3 k_{N,j} \wt{b}^{N,j}_1 \\
\cdots & \cdots & \cdots \\ \sum_{j=1}^3 k_{1,j} \wt{b}^{1,j}_{N-1} & \cdots & \sum_{j=1}^3 k_{N,j} \wt{b}^{N,j}_{N-1}
\end{array}
\right)
\end{equation} 
where $\wt{b}^{i,j}_{\ell}=\ell! \, b^{i,j}_{\ell}$. Genericity of the
coefficients $b^{i,j}_{\ell}$, in turn, ensures that $\wt{M_N}$ has nonzero determinant; and thus
   whenever $\ga \geq N$, the solution space of the system of equations on the right-hand side of \eqref{sys_eq_vt} is empty, i.e., $v_t(\sum_{i=1}^N \al_i m_i) \leq \mu+N-1$.

\medskip
\noindent Finally, assume more generally that
\[
\begin{split}
    v_t(m_1)&=\dots=v_t(m_{N_1})= \mu_1 \\    v_t(m_{N_1+1})&=\dots=v_t(m_{N_2})= \mu_2 \\
    &\cdots \\    v_t(m_{N_{\ga-1}+1})&=\dots=v_t(m_{N_{\ga}})= \mu_{\ga}
\end{split}
\]
in which $\mu_1<\mu_2<\cdots<\mu_{\ga}$ and $N_1<N_2<\cdots<N_{\ga}$ are positive integers with $N_{\ga}=N$. For each $i=1,\ldots, \ga$, let $L_i$ denote a linear combination 
of monomials $m_{N_{i-1}+1}, \dots, m_{N_i}$, where we set $N_0=0$. Using our analysis of the preceding paragraph, we have $v_t(L_1) \leq \mu_1+N_1-1$ and $v_t(L_i) \leq \mu_i+N_i-N_{i-1}$ for every $i=2,\dots,\ga$. Moreover, in order for the maximal valuation of $\sum_{i=1}^{\ga} L_i$ to be achieved, we must have $\mu_2=\mu_1+N_1-1$ and $\mu_i= \mu_{i-1}+N_{i-1}-N_{i-2}$ for every $i=3,\dots,\ga$.
 We now claim that 
$v_t(\sum_{i=1}^N L_i) \leq \mu_1+N-1$. Indeed, assume $L_1= \sum_{j=1}^{N_1} \al_j m_j$, $L_i=\sum_{j=N_{i-1}+1}^{N_i} \al_j m_j$ for every $j=2,\dots,\ga$, and that $\mu_2=\mu_1+N_1-1$ and $\mu_i= \mu_{i-1}+N_{i-1}-N_{i-2}$ for every $i=3,\dots,\ga$. Generalizing \eqref{sys_eq_vt}, we have 
$v_t(\sum_{i=1}^j L_i) \geq \mu_1+N$ if and only if
\begin{equation}\label{sys_eq_vt_general}
\begin{split}
\sum_{i=1}^{N_1} \al_i [t^{\mu_1+j}]m_i&=0 \text{ for }j=0,\dots,N_1-1 \\
\sum_{i=1}^{N_2} \al_i [t^{\mu_2+j}]m_i&=0 \text{ for }j=0,\dots,N_2-N_1-1 \\
&\cdots\\
\sum_{i=1}^{N_{\ga}} \al_i [t^{\mu_{\ga}+j}]m_i&=0 \text{ for }j=0,\dots,N_{\ga}-N_{\ga-1}-1.
\end{split}
\end{equation}
The coefficient matrix associated with the system of $N$ equations \eqref{sys_eq_vt_general}
is a concatenation of coefficient matrices of the form \eqref{generic_matrix_bis}, and as such
is of maximal rank. 
It follows that the system \eqref{sys_eq_vt_general} is insoluble, and therefore $v_t(\sum_{i=1}^N L_i) \leq \mu_1+N-1$. 
\end{proof}

\begin{coro}\label{cor_val1}
Whenever
$4d+\binom{d+2}{2} \le 2\ell$, there is no polynomial $G(f_1, f_2, f_3)$ with valuation satisfying
\[
d(2\ell+4)+\binom{d+2}{2} \le v_t(G(f_1, f_2, f_3))\le (d+1)2\ell-1.
\]
\end{coro}

\begin{proof}
Suppose that $G_0 = G_1 = \cdots = G_{k-1} = 0$ but $G_k \neq 0$.  Then $v_t(G) \ge k(2\ell)$.  If $k \ge d+1$, then clearly $v_t(G) \not\in [d(2\ell+4)+\binom{d+2}{2},(d+1)2\ell-1]$.  So suppose that $k \le d$. Lemma \ref{valuation_bound_bis} then implies that
$v_t(G_k) \le k (2\ell+4) + \binom{k+2}{2}$.
On the other hand, we have $G = G_k + \sum_{i \ge k+1} G_i$, while 
$v_t\left(\sum_{i \ge k+1} G_i \right) \ge (k+1) 2\ell$.
It follows immediately that 
$v_t(G) \in \left[k(2\ell), k  (2\ell+4)+\binom{k+2}{2}-1 \right]$.

\end{proof}

\begin{coro}\label{cor_val2}
Suppose ${\bf r} = (2\ell,2\ell+2,2\ell+4)$ for some $\ell \in \mb{N}_{>0}$.
The generic semigroup ${\rm S}(\bf{r})$ adapted to $\bf{r}$ has at least $2\ell-\left(4d+\binom{d+2}{2}\right)$ gaps in $[2d\ell, (d+1)2\ell]$.
\end{coro}
\begin{proof}
This follows directly from Corollary \ref{cor_val1} and the fact that $[d(2\ell+4)+\binom{d+2}{2},(d+1)2\ell-1]$ contains 
\[
\max\left\{0,
\left((d+1)2\ell-1\right) - \left(d(2\ell+4)+\binom{d+2}{2}\right)+1\right\}
\]
integers.
\end{proof}

\begin{coro}\label{genus_lb_generic_semigroup}
For any $k$, we have
$g\left({\rm S}(\bf{r})\right) \ge 2\ell(k+1) - 2k(k+1) - \binom{k+3}{3}$.
In particular, for any $\epsilon > 0$ and all sufficiently large $\ell$ we have
\[
g\left({\rm S}(\bf{r})\right) > \left(\frac{8}{3}-\epsilon\right) \ell^{3/2}.
\]
\end{coro}

\begin{proof}
Summing over the expression in Corollary \ref{cor_val2} from $d=0$ to $d=k$ shows that
\[
g\left({\rm S}(\bf{r})\right) \ge \sum_{d=0}^k \left(2\ell-\left(4d+\binom{d+2}{2}\right)\right).
\]
Noting that $\sum_{d=0}^k \binom{d+2}{2} = \binom{k+3}{3}$ completes the proof of the first inequality.
To prove the second inequality, we try to maximize the expression on the right-hand side of the statement of the corollary.  Accordingly, let $k =\lfloor2\sqrt{\ell}\rfloor$.  For all $\ell \gg 0$, we then have
\[
2\ell(k+1) \ge (4-\epsilon/3)\ell^{3/2}, 2k(k+1) \le (8+\epsilon/3)\ell, \text{ and }
\binom{k+3}{3} \le \left(\frac{4}{3}+\epsilon/3 \right) \ell^{3/2}.
\]
Likewise, $(8+\epsilon/3)\ell < (\epsilon/3) \ell^{3/2}$ whenever $\ell$ is sufficiently large. Combining these expressions completes the proof.
\end{proof}
We could prove slightly sharper inequalities by more carefully maximizing the cubic in $k$ on the right-hand side of the first statement of the corollary, but the additional benefit does not seem to be worth the more complicated statement.

\medskip
Using precisely the same strategy, we obtain a version of Corollary~\ref{genus_lb_generic_semigroup} that applies to arbitrary ramification triples ${\bf r} = (m,m+a,m+b)$ with $m \geq 2$ and $0<a<b$, provided $b$ is not too large relative to $m$.

\begin{prop}\label{general_lower_genus_bound}
Given positive integers $0<a<b$ and $m \geq 2$, let ${\bf r} = (m,m+a,m+b)$ and let ${\rm S}(\bf{r})$ be the generic semigroup adapted to $\bf{r}$. 
Suppose $f_1, f_2, f_3$ are generic power series in $t$ with $v_t(f_1) = m, v_t(f_2) = m+a,$ and $v_t(f_3) = m+b$. 
\begin{enumerate}[leftmargin=*]
\item Whenever
$bd+\binom{d+2}{2} \le m$, there is no polynomial $G(f_1, f_2, f_3)$ with valuation satisfying
\[
d(m+b)+\binom{d+2}{2} \le v_t(G(f_1, f_2, f_3))\le (d+1)m.
\]
\item ${\rm S}(\bf{r})$ has at least $m-\left(bd+\binom{d+2}{2}\right)$ gaps in $[dm, (d+1)m]$.
\item For any $k$, we have
\[
g\left({\rm S}(\bf{r})\right) \ge
m(k+1) - b \binom{k+1}{2} - \binom{k+3}{3}.
\]
\end{enumerate}
\end{prop}
In order to optimize the lower bound for $g\left({\rm S}(\bf{r})\right)$ in Proposition~\ref{general_lower_genus_bound}, one should choose $k \approx -b + \sqrt{b^2+m}$.

\begin{coro}\label{genus_lb_arith_progression_general_m}
Given positive integers $\ell$ and $m$, let ${\bf r}=(m\ell,m\ell+m,m\ell+2m)$ and let ${\rm S}({\bf r})$ be the
generic semigroup adapted to $\bf{r}$. For $\ell \gg m$, we have
\[
g({\rm S}({\bf r})) > \bigg(\frac{(2m)^{3/2}}{3}-\ep\bigg) \ell^{3/2}
\]
for every $\ep>0$.
\end{coro}

\begin{proof}
Choose $k=\lfloor \sqrt{2m \ell} \rfloor$, and argue as in the proof of Corollary~\ref{genus_lb_generic_semigroup}.
\end{proof}

\noindent Corollary~\ref{genus_lb_arith_progression_general_m} shows, in particular, that $g({\rm S}({\bf r}))$ is greater that $3m\ell+3m-6$ whenever $\ell \gg m$; as explained at the beginning of this section, it follows that $M^3_{d,g({\rm S}({\bf r}))}$ is reducible whenever $d \geq \max(2g({\rm S}({\bf r}))-2,n)$.

\begin{rem}\label{discrete_dynamical_system}
\emph{One might ask for {\it exact} structural results for generic semigroups ${\rm S}$ derived from ramification profiles $\bf{r}$, for example, in those cases where $\bf{r}$ is an arithmetic progression or a supersymmetric tuple. We conjecture that ${\rm S}$ may be obtained as a ``limit" $\lim_{i \ra \infty} {\rm S}^{(i)}$ of a discrete dynamical system, where ${\rm S}^{(0)}$ is the monoid minimally generated by the elements of ${\bf r}$ and ${\rm S}^{(i+1)}$ is the semigroup generated by ${\rm S}^{(i)}$ together with a number of additional insertions $\ell^b_1, \dots \ell^b_{N(b)-1}$ indexed by a certain subset of the \emph{Betti elements} $b \in B({\rm S}^{(i)},{\bf r})$; for a definition of these, see \cite[Sec. 2]{Ch}. Here $N(b)$ is the number of equivalence classes of factorizations of $b$ involving insertions to ${\rm S}^{(i-1)}$ and $\ell^b_i$ denotes the $i$-th gap of ${\rm S}^{(i)}$ greater than $b$, when these are ordered from smallest to largest.}\footnote{In this scheme, ${\rm S}^{(1)}$ by convention is the semigroup generated by the (monoid) generated by ${\bf r}$, together with insertions given by gaps greater than its Betti elements. In particular, when ${\bf r}$ is a triple in arithmetic progression, ${\rm S}^{(1)}$ becomes the semigroup ${\rm S}^*$ defined in the next subsection. 
}
\end{rem}

While it seems difficult to compute $\lim_{i \ra {\infty}} {\rm S}^{(i)}$ in general, we compute first-order approximations to these limits in the following subsection.

\subsection{First-order approximations and upper bounds on the genera of generic semigroups from arithmetic triples}

In this section we study value semigroups of generic cusps whose ramification profiles are triples that are in arithmetic progression.  We first consider the case where ${\bf r}=(2\ell,2\ell+2,2\ell+4)$ and then consider the more general case where ${\bf r}=(m\ell,m\ell+2,m\ell+4)$ for some $m \ge 2$.

\subsubsection{Case: $m=2$} 

\begin{thm}\label{m2_approx}
Let ${\bf r}=(2\ell,2\ell+2,2\ell+4)$, let $\rm{S}$ be the generic semigroup adapted to ${\bf r}$, and let
\[
\rm{S}^* := \begin{cases}
\langle 2\ell,2\ell+2,2\ell+4; 4\ell+5, (\ell+2)\ell+1 \rangle & \text{if } \ell \text{ is even}\\
\langle 2\ell,2\ell+2,2\ell+4; 4\ell+5, (\ell+3)\ell+1 \rangle& \text{if } \ell \text{ is odd}.
\end{cases}
\]
We have
(1) $\rm{S}^* \subseteq \rm{S}$; and (2) $g(\rm{S}) \le \left\lceil\frac{1}{2} \ell^2\right\rceil+2\ell$.
\end{thm}

\begin{proof}
Let $f_1(t)$ be a generic power series with $v_t(f_1) = 2\ell$, let $f_2(t)$ be a generic power series with $v_t(f_2) = 2\ell+2$, and let $f_3(t)$ be a generic power series with $v_t(f_3) = 2\ell+4$.  Suppose each power series has leading coefficient $1$.  In order to prove the first statement, note that because 
$f_1(t)f_3(t)$ and $f_2(t)^2$ are each of order $v_t=4\ell+4$, genericity of higher-order coefficients of the $f_i(t)$ guarantees that
$v_t\left(f_1(t)f_3(t)-f_2(t)^2\right) = 4\ell+5$.
Similarly, for $\ell$ even we have 
\[
v_t\left(f_1(t)^{\frac{\ell+2}{2}} - f_3(t)^{\frac{\ell}{2}}\right) = (\ell+2)\ell+1
\]
while for $\ell$ odd we have 
\[
v_t\left(f_1(t)^{\frac{\ell+3}{2}} - f_2(t) f_3(t)^{\frac{\ell-1}{2}}\right) = (\ell+3)\ell+1.
\]

For the second item, note that $\rm{S}^* \subseteq \rm{S}$ implies that $g(\rm{S}) \le g(\rm{S}^*)$.  We prove that $g(\rm{S}^*) \le \left\lceil\frac{1}{2} \ell^2\right\rceil+2\ell$ by explicitly computing the set of gaps of $\rm{S}^*$.  It is straightforward  to check that the set of gaps of $\langle 2\ell,2\ell+2,2\ell+4; 4\ell+5, (\ell+2)\ell+1 \rangle$ is given by 
{\small
\[
\begin{split}
{\rm G}_0&=\{1,\dots,2\ell-1\} \bigsqcup \{2\ell+1, 2\ell+3,\dots,4\ell+3\} \bigsqcup \bigsqcup_{i=1}^{\frac{\ell}{2}-1} 2\{i\ell+2i+1,\dots,(i+1)\ell-1\}\\
& \bigsqcup \bigsqcup_{i=1}^{\frac{\ell}{2}-2} \{2(i+1)\ell+2(2i+1)+1, 2(i+1)\ell+2(2i+1)+3, \dots, 2(i+2)\ell+3\} \\
& \bigsqcup \{\ell^2+ 2\ell-1, \ell^2+2\ell+3\}.
\end{split}
\]
}
Similarly, it is straightforward to check that the set of gaps of $\langle 2\ell,2\ell+2,2\ell+4; 4\ell+5, (\ell+3)\ell+1 \rangle$ is given by 
{\small
\[
\begin{split}
    {\rm G}_1 &= \{1,\dots,2\ell-1\} \bigsqcup \{2\ell+1, 2\ell+3,\dots,4\ell+3\} \bigsqcup \bigsqcup_{i=1}^{\frac{\ell-1}{2}-1} 2\{i\ell+2i+1,\dots,(i+1)\ell-1\}\\
    & \bigsqcup \bigsqcup_{i=1}^{\frac{\ell-1}{2}-1} \{2(i+1)\ell+2(2i+1)+1, 2(i+1)\ell+2(2i+1)+3, \dots, 2(i+2)\ell+3\}\\
   &  \bigsqcup \{\ell^2+3\ell+3\}.
\end{split}
\]
}
By counting these gaps, we see that $g(\rm{S}^*) = \left\lceil\frac{1}{2} \ell^2\right\rceil+2\ell$.
\end{proof}

We have just seen that $g(\rm{S}) \le \lceil \frac{1}{2}\ell^2 \rceil + 2\ell$.  In Corollary \ref{genus_lb_generic_semigroup}, we saw that for all sufficiently large $\ell$ we have $g(\rm{S})$ is at least a constant times $\ell^{3/2}$.  It seems likely that for large values of $\ell$ the lower bound is closer to the true genus than the upper bound, but we do not pursue this further here.

\subsection{Case: $m \geq 2$}
Easy adaptations of the arguments used to prove the first statement in Theorem~\ref{m2_approx}  
yield the following more general containment.

\begin{thm}\label{generic_semigroup_consecutive_multiples}
Fix positive integers $m \ge 2$ and $\ell \ge 2m$ and let ${\bf r}(m,\ell)=(m\ell,m\ell+m,m\ell+2m)$.  Let $\rm{S}$ be the generic semigroup adapted to ${\bf r}(m,\ell)$. If $\ell$ is even, let
\[
\rm{S}^* =
\langle m\ell,m\ell+m,m\ell+2m; 2m(\ell+1)+1, m\left(\frac{\ell}{2}+1\right)\ell+1 \rangle
\]
while if $\ell$ is odd, let
\[
\rm{S}^* =
\left\langle m\ell,m\ell+m,m\ell+2m; 2m(\ell+1)+1,\frac{m}{2}(\ell+1)(\ell+2)+1, \frac{m}{2}\ell(\ell+3)+1 \right\rangle.
\]
Then
\begin{enumerate}
\item  $\rm{S}^* \subseteq \rm{S}$.
\item Whenever $\ell$ is even, we have 
\[
g(\rm{S}) \le \frac{1}{4}m \ell^2+ m(m-1)\ell+ (m-1)(m-2).
\]
\item Whenever $\ell$ is odd, we have 
\[
g(\rm{S}) \le \frac{1}{4}m(\ell+1)(\ell-2)+ m(m-1)\ell+ (m-1)(m-2).
\]
\end{enumerate}
\end{thm}

We prove the first part of this theorem and return to the second part later in this section.
\begin{proof}[Proof of Theorem~\ref{generic_semigroup_consecutive_multiples} (1)]
As in the proof of Theorem~\ref{m2_approx}, let $f_1(t),f_2(t),f_3(t)$ be power series with $t$-valuations $m\ell, m \ell+m$, and $m\ell+2m$, respectively, and suppose that each $f_i$ has leading coefficient $1$.  Since $f_1(t) f_3(t)$ and $f_2(t)^2$ each have $t$-valuation $2(m\ell+m)$, 
genericity ensures that
$
v_t\left(f_1(t)f_3(t)-f_2(t)^2\right) = 2m(\ell+1)+1$.
Furthermore, 
we have
\[
\left(\frac{\ell}{2}+1\right) m\ell = \frac{\ell}{2}\left(m\ell+2m\right)
\]
whenever $\ell$ is even; while
\[
\left(\frac{\ell+1}{2} \right)(m\ell + 2m) = \left(\frac{\ell+1}{2} \right)(m\ell) + 1(m\ell+m)
\]
and 
\[
\left(\frac{\ell+3}{2}\right)m\ell =1(m\ell+m)+ \left( \frac{\ell-1}{2} \right)(m\ell+2m)
\]
whenever $\ell$ is odd.
These factorizations in $\langle m\ell , m\ell +m, m\ell  + 2m\rangle$ yield pairs of monomials 
in $f_1(t), f_2(t), f_3(t)$ with the same $t$-valuation. In each case the genericity of higher-order coefficients implies that the difference is a power series with $t$-valuation one larger.
 
\end{proof}

In order to complete our generalization of Theorem~\ref{m2_approx}
to the case where $m \geq 3$, we will write down the gap set ${\rm G}= \mb{N} \setminus {\rm S}^*$ explicitly.  For this purpose, it suffices to compute the {\it Ap\'ery set} of ${\rm S^*}$. 
The \emph{Ap\'ery set} of a numerical semigroup ${\rm S}$ with respect to an element $n \in {\rm S}$ is 
$\Ap({\rm S};n) = \{x \in {\rm S}\colon x-n \not\in {\rm S}\}$.
In general, $\Ap({\rm S};n)$ consists of $0$ together with $n-1$ positive integers, each in a distinct residue class modulo $n$.  In this paper, we will only consider the Ap\'ery set of a numerical semigroup with respect to its \emph{multiplicity}, its smallest nonzero element.  For $i \in \{1,2,\ldots, m\ell -1\}$, let $e_i$ denote the smallest $x \in \rm{S}^*$ with $x\equiv i \pmod{m\ell}$.  The number of gaps $x\in \rm{G}$ satisfying $x \equiv i \pmod{m\ell}$ is $\frac{e_i - i}{m\ell}$.  In this way, we see that 
\begin{equation}\label{gS*}
g(\rm{S}^*) = \sum_{i=1}^{m\ell-1} \frac{e_i-i}{m\ell}.
\end{equation}
The goal of the rest of this section is to prove the following characterization of $(e_1,\ldots, e_{m\ell-1})$.
\begin{prop}\label{Apery_S*}
Suppose $\ell$ is even and let ${\rm S}^*$ be defined as in Theorem \ref{generic_semigroup_consecutive_multiples}.  For $i \in \{1,2,\ldots,m\ell-1\}$, let $e_i$ be the smallest element of ${\rm S}^*$ congruent to $i$ modulo~$m\ell$.  
\begin{itemize}[leftmargin=*]
\item For $k \in\{0,\ldots,m-1\}$ and $j \in \{k,k+1,\ldots, \ell-1\}$, we have
\begin{equation}\label{Apery_set_nonspecial}
\begin{split}
e_{2mj+k} &= (j-k)(\ell m+ 2m)+ k(2m(\ell+1)+1) \text{ and}\\
e_{2mj+m+k} &=  (\ell m+m)+ (j-k)(\ell m+ 2m)+ k(2m(\ell+1)+1).
\end{split}
\end{equation}
\item For $j \in \{0,\ldots, m-2\}$, we have
\begin{equation}\label{Apery_set_special}
e_{2mj+(j+1)}=j(2m(\ell+1)+1)+(m\ell(\ell/2+1)+1).
\end{equation}
For $j \in \{0,\ldots, m-2\}$ and $k \in \{j+2,j+3,\ldots, m-1\}$, we have
\begin{equation}\label{Apery_set_special_bis}
e_{2mj+k}= \left(j+\frac{\ell}{2}-k\right)(m\ell+2m)+ k(2m(\ell+1)+1).
\end{equation}
\end{itemize}
\end{prop}

\begin{prop}\label{Apery_S*bis}
Suppose $\ell$ is odd and let ${\rm S}^*$ be defined as in Theorem \ref{generic_semigroup_consecutive_multiples}.  For $i \in \{1,2,\ldots,m\ell-1\}$, let $e_i$ be the smallest element of ${\rm S}^*$ congruent to $i$ modulo~$m\ell$.  
\begin{itemize}[leftmargin=*]
\item For $k \in\{0,\ldots,m-1\}$ and $j \in \{k,k+1,\ldots, \ell-1\}$, we have
\[
\begin{split}
e_{2mj+k} &= (j-k)(\ell m+ 2m)+ k(2m(\ell+1)+1) \text{ and}\\
e_{2mj+m+k} &=  (\ell m+m)+ (j-k)(\ell m+ 2m)+ k(2m(\ell+1)+1).
\end{split}
\]
\item For $j \in \{0,\ldots, m-2\}$, we have
\[
e_{2mj+(j+1)}= \frac{m}{2} \ell(\ell+3)+1+ j(2m(\ell+1)+1).
\]
\item For $j \in \{0,\ldots, m-2\}$ and $k \in \{j+2,j+3,\ldots, m-1\}$, we have
\[
e_{2mj+k}= (m\ell+m)+\left(j+\frac{\ell-1}{2}-k\right)(m\ell+2m)+ k(2m(\ell+1)+1).
\]
\end{itemize}
\end{prop}

\noindent Propositions~\ref{Apery_S*} and \ref{Apery_S*bis} give a complete characterization of $(e_1,e_2,\ldots, e_{m\ell-1})$.

Assuming Proposition~\ref{Apery_S*} for now, we prove the second part of Theorem~\ref{generic_semigroup_consecutive_multiples}.  
\begin{proof}[Proof of Theorem~\ref{generic_semigroup_consecutive_multiples} (2)] 
As in the proof of Theorem~\ref{m2_approx}, the fact that $\rm{S}^* \subseteq \rm{S}$ implies $g(\rm{S}) \le g(\rm{S}^*)$. We now use \eqref{gS*} to compute $g(\rm{S}^*)$.

\medskip
\noindent Suppose $\ell$ is even.  The contribution of the Ap\'ery values in \eqref{Apery_set_nonspecial} to $g({\rm S}^*)$ is
\[
\begin{split}
g_1&= \sum_{k=0}^{m-1} \sum_{j=k}^{\frac{\ell}{2}-1} \frac{e_{2mj+k}-(2mj+k)}{m \ell}+ \sum_{k=0}^{m-1} \sum_{j=k}^{\frac{\ell}{2}-1} \frac{e_{2mj+m+k}-(2mj+m+k)}{m \ell} \\
&= \sum_{k=0}^{m-1} \sum_{j=k}^{\frac{\ell}{2}-1} (j+k)+ \sum_{k=0}^{m-1} \sum_{j=k}^{\frac{\ell}{2}-1} (j+k+1) \\
&= \frac{m \ell^2}{4}+ \frac{m \ell}{2}(m-1)- m(m-1)(m-1/2).
\end{split}
\]
\noindent Similarly, the contribution of the Ap\'ery values in \eqref{Apery_set_special}
to $g({\rm S}^*)$ is
\[
\begin{split}
    g_2 &= \sum_{j=0}^{m-2} \frac{e_{2mj+j+1}-(2mj+j+1)}{m \ell}+ \sum_{j=0}^{m-2} \frac{e_{2mj+j+1+m}-(2mj+j+1+m)}{m \ell} \\
    &= \sum_{j=0}^{m-2} (2j+ \frac{\ell}{2}+1)+ \sum_{j=0}^{m-2} (2j+ \frac{\ell}{2}+2) \\
    &= (m-1)(\ell+2m-1).
\end{split}
\]
The contribution of the Ap\'ery values in \eqref{Apery_set_special_bis} to $g({\rm S}^*)$ is
\[
\begin{split}
    g_3 &= \sum_{j=0}^{m-2} \sum_{k=j+2}^{m-1} \frac{e_{2mj+k}-(2mj+k)}{m\ell}+ \sum_{j=0}^{m-2} \sum_{k=j+2}^{m-1} \frac{e_{2mj+k+m}-(2mj+k+m)}{m\ell} \\
    &= \sum_{j=0}^{m-2} \sum_{k=j+2}^{m-1} (j+k+\frac{\ell}{2}+1)+ \sum_{j=0}^{m-2} \sum_{k=j+2}^{m-1} (j+k+\frac{\ell}{2}+2)\\
    &= \frac{\ell}{2}(m-1)(m-2)+(m-1)(m+1/2)(m-2).
\end{split}
\]
It follows that $g({\rm S}^*)= g_1+g_2+g_3=\frac{1}{4}m \ell^2+ m(m-1)\ell+ (m-1)(m-2)$.

\medskip
Now suppose $\ell$ is odd.  
This time, the
Ap\'ery elements $e_{2mj+k}$ and $e_{2mj+m+k}$ with $j \geq k$ 
contribute
\[
\begin{split}
g_1&= \sum_{k=0}^{m-1} \sum_{j=k}^{\frac{\ell-1}{2}} \frac{e_{2mj+k}-(2mj+k)}{m \ell}+ \sum_{k=0}^{m-1} \sum_{j=k}^{\frac{\ell-3}{2}} \frac{e_{2mj+m+k}-(2mj+m+k)}{m \ell}\\
&= \sum_{k=0}^{m-1} \sum_{j=k}^{\frac{\ell-1}{2}} (j+k)+ \sum_{k=0}^{m-1} \sum_{j=k}^{\frac{\ell-3}{2}} (j+k+1) \\
&= \frac{m \ell^2}{4}+ \frac{m\ell}{2}(m-1)- m((m-1/2)(m-1)+1/4)
\end{split}
\]
to $g({\rm S}^*)$.
The Ap\'ery elements $e_{2mj+j+1}$ and $e_{2mj+m+j+1}$ with $0 \leq j \leq m-2$ contribute
\[
\begin{split}
    g_2 &= \sum_{j=0}^{m-2} \frac{e_{2mj+j+1}-(2mj+j+1)}{m \ell}+ \sum_{j=0}^{m-2} \frac{e_{2mj+j+1+m}-(2mj+j+1+m)}{m \ell} \\
    &= 2\sum_{j=0}^{m-2} (2j+ \frac{\ell}{2}+1)\\
    &= (m-1)(\ell+2m-1).
\end{split}
\]
The remaining Ap\'ery elements contribute
\[
\begin{split}
    g_3 &= \sum_{j=0}^{m-2} \sum_{k=j+2}^{m-1} \frac{e_{2mj+k}-(2mj+k)}{m\ell}+ \sum_{j=0}^{m-2} \sum_{k=j+2}^{m-1} \frac{e_{2mj+k+m}-(2mj+k+m)}{m\ell} \\
    &= 2\sum_{j=0}^{m-2} \sum_{k=j+2}^{m-1} (j+k+\frac{\ell+3}{2})\\
    &=\frac{\ell}{2}(m-1)(m-2)+(m-1)(m+1/2)(m-2).
    \end{split}
\]
It follows that $g({\rm S}^*)= g_1+g_2+g_3=\frac{m(\ell+1)(\ell-1)}{4}+ m(m-1)\ell+ (m-1)(m-2)$.

\end{proof}

It remains to prove Propositions \ref{Apery_S*} and \ref{Apery_S*bis}. We give a full proof of Proposition \ref{Apery_S*}
and leave the proof of the similar Proposition~\ref{Apery_S*bis} to the reader. 
For the remainder of this section, suppose $\ell$ is even.  Proposition \ref{Apery_S*} follows from the following three lemmas.
\begin{lemma}\label{Ap1}
Let $m \ge 2$.  Suppose that $\ell$ is an even integer satisfying $\ell \ge 2m$. Let ${\rm T} = \langle m \ell, m \ell+m, m\ell+2m ,2m(\ell+1)+ 1 \rangle$. Then $\Ap({\rm T};m \ell)$ consists of those elements of the form 
\[
x  (m\ell+m) + y  (m\ell+2m) + z  (2m (\ell+1)+1)
\]
where $x \in \{0,1\},\ y \in \{0,1,\ldots, \frac{\ell}{2}-1\}$, and $z \in \{0,1,\ldots, m-1\}$. 
\end{lemma}

\begin{lemma}\label{ApSSprime}
Continuing with the notation of Lemma~\ref{Ap1}, suppose $s \in \Ap({\rm T}; m\ell)$.  Then $s \in \Ap({\rm S}^*;m\ell)$ if and only if there does not exist a positive integer $k$ for which 
\[
s- k  \left(\frac{\ell(\ell+2)}{2} m + 1\right) -m\ell \in {\rm T}.
\]
\end{lemma}

\begin{lemma}\label{ApSprime}
Let $m \ge 2$.  Suppose that $\ell$ is an even integer satisfying $\ell \ge 2m$. Then $\Ap({\rm S}^*;m\ell)$ consists of the elements
\[
x  (m\ell+m) + y  (m\ell+2m) + z  (2m (\ell+1)+1) 
\]
with $x \in \{0,1\},\ y \in \{0,1,\ldots, \frac{\ell}{2}-2\}$, and $z\in \{0,1,\ldots, m-1\}$, together with the elements
\[
\left(\frac{\ell}{2}-1\right)(m\ell+2m), (m\ell+m) + \left(\frac{\ell}{2}-1\right)(m\ell+2m)
\]
and the elements
\[
x  (m\ell+m) + z  (2m (\ell+1)+1) + \left(\frac{\ell(\ell+2)}{2} m + 1\right)
\]
where $x\in \{0,1\}$ and $z\in \{0,1,\ldots, m-2\}$.
\end{lemma}

\noindent We now prove these three lemmas.
\begin{proof}[Proof of Lemma~\ref{Ap1}]
For every $s \in \Ap({\rm T};m\ell)$, we have $s - m\ell \not\in {\rm T}$.  Every $s \in {\rm T}$ is of the form
\begin{equation}\label{eq:s-factorization}
s = w  m\ell + x  (m\ell+m) + y  (m\ell+2m) + z  (2m (\ell+1)+1)
\end{equation}
where $w,x,y,z \in \Z_{\ge 0}$. Accordingly, suppose $(w,x,y,z) \in Z(s)$ is a factorization of $s \in {\rm T}$.  If $w \ge 1$ then $s-m\ell \in {\rm T}$.  Therefore, every $s \in \Ap({\rm T};m\ell)$ is as in \eqref{eq:s-factorization} with $w=0$. On the other hand, the fact that 
\[
m   (2m (\ell+1)+1)=  1 (m\ell+m) + m  (m\ell+2m) + (m-1) m\ell
\]
implies that $z-m\ell \in {\rm T}$ whenever 
$z \ge m$. Similarly, the two equalities
\[
2  (m\ell+m) =  1 (m\ell+2m) + m\ell \text{ and }
\frac{\ell}{2}  (m\ell+2m) =  \left(\frac{\ell}{2}+1\right) m\ell
\]
imply that $s-m\ell \in {\rm T}$ whenever $x \ge 2$ or $y \ge \frac{\ell}{2}$.
We conclude that every $s \in \Ap({\rm T};m\ell)$ is of the form
\[
s = x  (m\ell+m) + y  (m\ell+2m) + z  (2m (\ell+1)+1) 
\]
with $x \in \{0,1\},\ y \in \{0,1,\ldots, \frac{\ell}{2}-1\}$, and $z \in \{0,1,\ldots, m-1\}$.  As there are precisely $m\ell$ elements of this form, they necessarily comprise $\Ap({\rm T};m\ell)$.
\end{proof}

\begin{proof}[Proof of Lemma~\ref{ApSSprime}]
As $s \in \Ap({\rm T};m \ell)$, we have $s - m\ell \not\in {\rm T}$.  By definition, $s \in \Ap({\rm S}^*;m\ell)$ if and only if $s - m \ell \not\in {\rm S}^*$. Here 
$s-m\ell \in {\rm S}^*$ if and only if
\[
s-m\ell =  x  (m\ell+m) + y  (m\ell+2m) + z  (2m (\ell+1)+1) + k  \left(\frac{\ell(\ell+2)}{2} m + 1\right)
\]
where $x,y,z,k \in \Z_{\ge 0}$. As $s- m\ell \not\in {\rm T}$, there are no such $x,y,z,k$ with $k=0$. Consequently, $s-m\ell \in {\rm S}^*$ if and only if there exists a positive integer $k$ for which
\[
s- k  \left(\frac{\ell(\ell+2)}{2} m + 1\right) -m\ell \in {\rm T}.
\]
\end{proof}

\begin{proof}[Proof of Lemma~\ref{ApSprime}]
When $(m,\ell) = (2,4)$ we can check directly that the statement holds, so from now on suppose that either $m\ge 4$ or $m=2$ and $\ell\ge 6$. To begin, note that
\[
\left(\frac{\ell}{2}-1\right)(m\ell+2m)+ (2m (\ell+1)+1) - m\ell = \frac{\ell(\ell+2)}{2} m + 1.
\]
Therefore, $s \not\in \Ap({\rm S}^*;m\ell)$ whenever
\[
s = x  (m\ell+m) + y  (m\ell+2m) + z  (2m (\ell+1)+1) 
\]
with $y = \frac{\ell}{2} -1$ and $z \in \Z_{\ge 1}$.

\medskip
Next we show that every $s \in \Ap({\rm T};m\ell)$ of the form
\begin{equation}\label{eq:most_in_Ap}
s = x  (m\ell+m) + y  (m\ell+2m) + z  (2m (\ell+1)+1) 
\end{equation}
with $x \in \{0,1\},\ y \in \{0,1,\ldots, \frac{\ell}{2}-1\}$, $z\in \{0,1,\ldots, m-1\}$, and such that $z=0$ whenever $y = \frac{\ell}{2}-1$, belongs to $\Ap({\rm S}^*;m\ell)$. According to Lemma~\ref{ApSSprime}, we need only show that there is no positive integer $k$ for which
\[
s- k  \left(\frac{\ell(\ell+2)}{2} m + 1\right) -m\ell \in {\rm T}.
\]

\noindent We first show that there is no such $k$ for which $z < k$.  We will use the fact that 
\[
(2m (\ell+1)+1)  - \left(\frac{\ell(\ell+2)}{2} m + 1\right) = m\ell + 2m - \frac{\ell^2 m}{2}.
\]
If $z < k$, we have
\begin{equation}\label{eq:z<k}
\begin{split}
& x  (m\ell+m) + y  (m\ell+2m) + z  (2m (\ell+1)+1)- k  \left(\frac{\ell(\ell+2)}{2} m + 1\right) -m\ell\\
\le \hspace{3pt}& (m\ell+m) + \left( \frac{\ell}{2}-1\right)(m\ell+2m) + k  \left(m\ell + 2m - \frac{\ell^2 m}{2}\right) - (2m(\ell+1)+1) - m\ell.
\end{split}
\end{equation}
Since $m\ell + 2m - \frac{\ell^2 m}{2}$ is negative the right-hand side of \eqref{eq:z<k} is a decreasing function of $k$.  For $k = 1$ we see that this quantity is negative.  This completes the proof in this case.

\noindent Now suppose $z \ge k$ and that $k$ is a positive integer such that 
\[
s- k  \left(\frac{\ell(\ell+2)}{2} m + 1\right) -m\ell \in {\rm T}.
\]
Then there exist  $a,b,c \in \Z_{\ge 0}$ for which 
\[
\begin{split}
&x  (m\ell+m) + y  (m\ell+2m) + z  (2m (\ell+1)+1)- k  \left(\frac{\ell(\ell+2)}{2} m + 1\right) -m\ell \\
= \hspace{3pt}&a  (m\ell+m) + b  (m\ell+2m) + c  (2m (\ell+1)+1).
\end{split}
\]

Taking both sides of 
this equation modulo $m$, we see that $z-k \equiv c\pmod{m}$. This implies that $c \ge z-k$.
Note that 
\[
z  (2m (\ell+1)+1)- k  \left(\frac{\ell(\ell+2)}{2} m + 1\right)-  (z-k)(2m (\ell+1)+1)  
= 
k  \left(m\ell + 2m - \frac{\ell^2 m}{2}\right).
\]
Since $a,b \in \Z_{\ge 0}$, we must have 
\[
x  (m\ell+m) + y  (m\ell+2m) + k  \left(m\ell + 2m - \frac{\ell^2 m}{2}\right) - m\ell \ge 0.
\]
Since $x\in\{0,1\}$ and $y \in \{0,1,\ldots, \frac{\ell}{2}-1\}$, we have
\[
\begin{split}
&x  (m\ell+m) + y  (m\ell+2m) + k  (2m (\ell+1)+1)- k  \left(\frac{\ell(\ell+2)}{2} m + 1\right) -m\ell \\
\le \hspace{3pt}& (m\ell+m) + \left( \frac{\ell}{2}-1\right)(m\ell+2m) + k  \left(m\ell + 2m - \frac{\ell^2 m}{2}\right) - m\ell.
\end{split}
\]
Recalling that either $m \ge 4$ or $m=2$ and $\ell \ge 6$, since 
$\ell \ge 2m$, this expression is negative when $k \ge 2$.  When $k=1$ this expression is equal to $m\ell+m$.  We conclude that if there exists $k \in \Z_{\ge 1}$ for which
\[
x  (m\ell+m) + y  (m\ell+2m) + z  (2m (\ell+1)+1) -k  \left(\frac{\ell(\ell+2)}{2} m + 1\right) -m\ell \in {\rm T}
\]
where $x\in \{0,1\}, y \in \{0,1,\ldots, \frac{\ell}{2}-1\}$, and $z\in \{0,1,\ldots, m-1\}$, then $y = \frac{\ell}{2}-1$ and $z \neq 0$.  This completes the proof that the elements described in \eqref{eq:most_in_Ap} belong to $\Ap(S^*;m\ell)$.

\medskip
Now suppose that 
\[
s = x  (m\ell+m) + z  (2m (\ell+1)+1) + \left(\frac{\ell(\ell+2)}{2} m + 1\right)
\]
where $x\in \{0,1\}$ and $z\in \{0,1,\ldots, m-2\}$. We see that $s+m\ell \in \Ap({\rm T};m\ell)$.  To complete the proof, we need only show that $s-m \ell \not\in {\rm S}^*$.  

\noindent Accordingly, suppose $s-m\ell \in {\rm S}^*$.  There are then $a,b,c,d \in \Z_{\ge 0}$ for which
\begin{equation}\label{(a,b,c,d)}
\begin{split}
& x  (m\ell+m) + z  (2m (\ell+1)+1) + \left(\frac{\ell(\ell+2)}{2} m + 1\right) -m\ell \\
 = \hspace{3pt}& a  (m\ell+m) + b  (m\ell+2m) + c  (2m (\ell+1)+1) + d  \left(\frac{\ell(\ell+2)}{2} m + 1\right).
\end{split}
\end{equation}
The fact that $s+m\ell \in \Ap({\rm S}^*;m\ell)$ now forces $d \ge 1$.  If there were such a solution $(a,b,c,d)$ with $d=1$, we would then have 
\[
x  (m\ell+m) + z  (2m (\ell+1)+1)-m\ell \in {\rm T}
\]
which contradicts the fact that $x  (m\ell+m) + z  (2m (\ell+1)+1) \in \Ap({\rm T};m\ell)$.  So necessarily $d\ge 2$.  Taking both sides of \eqref{(a,b,c,d)} modulo $m$, we deduce that $z \equiv c + (d-1) \pmod{m}$.  As $z \in \{0,1,\ldots, m-2\}$ and $c, d-1 \ge 0$, we have $z-c \le d-1$. The fact that
\[
x  (m\ell+m) +(d-1) \bigg(2m(\ell+1)+1 -\bigg(\frac{\ell(\ell+2)}{2} m + 1\bigg)  \bigg) - m \ell < 0
\]
allows us to conclude.
\end{proof}

\section{Semigroups from supersymmetric triples}\label{semigroups_from_supersymmetric_triples}

\noindent In this section we suppose that $2 \le a< b< c$ are pairwise relatively prime integers, and we let  ${\rm S}= {\rm S}(a,b,c) := \langle ab,ac,bc\rangle$. We recall some basic facts about these numerical semigroups.

\begin{prop}\label{supersym_prop}
Suppose $2 \le a< b< c$ are pairwise relatively prime integers. Let  ${\rm S}= {\rm S}(a,b,c) := \langle ab,ac,bc\rangle$.  We have that
\begin{enumerate}[leftmargin=*]
\item $F(\rm{S}) = 2abc-(ab+ac+bc)$;
\item $\rm{S}$ is symmetric; and
\item $g(\rm{S}) = \frac{F(\rm{S)}+1}{2} = abc - \frac{ab+ac+bc-1}{2}$.
\end{enumerate}
\end{prop}

\subsection{Supersymmetric Severi varieties of excess dimension}

In this subsection, we study generalized Severi varieties $M^3_{d,g;{\rm S},{\bf r}}$ associated with ${\rm S}=\langle ab,bc,ac \rangle$ and ${\bf r}=(ab,ac,bc)$.  In recent work \cite{CLMR}, Cotterill, Lima, Martins, and Reis have computed the dimension of this variety building on a conjectural combinatorial model for Severi dimensions introduced in \cite{CLM}.
\begin{thm}\cite[Theorem 4.2]{CLMR}
Assume that $d \ge 2g(\rm{S})$; then
\[
{\rm cod}(M^3_{d,g;{\rm S},{\bf r}},M^3_d) = 2\rho(abc)+ab+ac+bc-7,
\]
where $\rho(abc)$ denotes the number of gaps in $\rm{S}$ strictly larger than $abc$.
\end{thm}
Comparing this formula 
against
the codimension of $g$-nodal rational curves of degree $d$ in $\mathbb{P}^3$ yields the following corollary.\footnote{Recall that the codimension of the $g$-nodal locus in $M^n_d$ is $(n-2)g$ whenever $d$ is sufficiently large relative to $g$.}
\begin{coro}\label{severi_cor}
Assume $d \geq 2g$. The Severi variety $M^3_{d,g;{\rm S},{\bf r}}$ is excess-dimensional whenever
\begin{equation}\label{rhobound1}
\rho(abc)< \frac{abc}{2}- \frac{3}{4}(ab+ac+bc)+ \frac{15}{4}.
\end{equation}
\end{coro}
\noindent We will show that \eqref{rhobound1} holds for most triples $(a,b,c)$, in order to deduce the following.
\begin{thm}\label{excess_thm}
Let $2 \le a< b< c$ be pairwise relatively prime integers, let  ${\rm S}= {\rm S}(a,b,c) := \langle ab,ac,bc\rangle$, and assume that $d \ge 2g(\rm{S})$. The Severi variety $M^3_{d,g;{\rm S},{\bf r}}$ is excess-dimensional whenever $4 \le a < b < c$ and $(a,b,c) \neq (4,5,7)$.
\end{thm}

A key ingredient here is an upper bound for $\rho(abc)$, obtained by re-interpreting it as the number of lattice points in a certain right-angled rational simplex in $\mathbb{R}^3$.
\begin{lemma}\label{simplex_lem}
Let $2 \le a< b< c$ be pairwise relatively prime integers and let  ${\rm S}= {\rm S}(a,b,c) := \langle ab,ac,bc\rangle$.  Let $\rho(abc)$ denote the number of gaps in $\rm{S}$ strictly larger than $abc$.  Then $\rho(abc)$ is equal to the number of lattice points in the rational simplex $\Delta(\alpha,\beta,\gamma)$ with vertices $(0,0,0), (\alpha,0,0), (0,\beta,0), (0,0,\gamma)$ where 
\[
\al=c-1 - \frac{c}{a} -\frac{c}{b}, \be=b-1 - \frac{b}{a} - \frac{b}{c}, \text{ and } \ga=a-1 - \frac{a}{b} - \frac{a}{c}.
\]
\end{lemma}

\begin{proof}
According to Proposition \ref{supersym_prop}, $\rm{S}$ is symmetric; it follows that gaps of ${\rm S}$ greater than $abc$ are in bijection with elements of ${\rm S}$ less than $F({\rm S})-abc= abc- (ab+ac+bc)$. Accordingly, 
$\rho(abc) = \#\{x \in {\rm S}\colon x < abc-(ab+ac+bc)\}$.
On the other hand, the {\it factorization map} $\varphi\colon \mathbb{Z}_{\ge 0}^3 \rightarrow \mathbb{Z}_{\ge 0}$ given by
$\varphi(u,v,w) = u ab + v ac + w bc$
defines a bijection between triples $(u,v,w) \in \mathbb{Z}_{\ge 0}^3$ and factorizations of elements of ${\rm S}$.  It is easy to check that every element $s < abc$ in ${\rm S}$ has a unique factorization.
Since the lattice points of $\Delta(\al,\be,\ga)$ index factorizations of elements of ${\rm S}$ less than $abc-(ab+ac+bc)$, there are precisely $\rho(abc)$ of these. 
\end{proof}

Recall that we require $2\le a< b<c$ to be pairwise relatively prime. Note that $\gamma < 1$ if and only if $a = 2$, or $a = 3$ and $(b,c) \in\{(4,5),(4,7),(4,11),(5,7)\}$.  In these cases, any lattice point in 
$\Delta(\al,\be,\ga)$ must have final coordinate equal to $0$; so $\rho(abc)$ is equal to the number of lattice points in the rational triangle in the plane with vertices $(0,0),(\alpha,0),(0,\beta)$.  We  
may bound the number of lattice points in this triangle from above by finding a triangle with vertices at lattice points in which it is contained and then applying Pick's theorem; we leave the details to the reader.  
For simplicity, we suppose that $a\ge 4$ for the remainder of this section.

\medskip
Lemma \ref{simplex_lem} produces an upper bound for $\rho(abc)$ 
when used in tandem with a theorem of Yau and Zhang \cite{YZ}.  
Given real numbers $a_1 \ge a_2 \ge a_3 \ge 1$, let $P_{(a_1,a_2,a_3)}$ (resp., $Q_{(a_1,a_2,a_3)}$) denote the number of strictly positive (resp., nonnegative) integral triples $(x_1,x_2,x_3)$ for which 
$\frac{x_1}{a_1} + \frac{x_2}{a_2} + \frac{x_3}{a_3} \le 1$.
According to Lemma \ref{simplex_lem}, we have  
$\rho(a,b,c) = Q_{(\alpha,\beta,\gamma)}$.  Now suppose that we are not in one of the exceptional cases 
of the preceding paragraph
in which 
$\gamma < 1$.  
Yau and Zhang note that attaching a unit cube to the right of and above each lattice point of $\Delta(\alpha,\beta,\gamma)$ 
produces the following upper bound.
\begin{prop}\cite[p. 912]{YZ}\label{weakYZ}
Suppose $\alpha \ge \beta \ge \gamma \ge 1$ are real numbers, and let $\eta = \frac{1}{\alpha} + \frac{1}{\beta} + \frac{1}{\gamma}$.  We have
$
Q_{(\alpha,\beta,\gamma)} \le \frac{\alpha\beta\gamma}{6} (1+\eta)^3$.
\end{prop}

This upper bound implies Theorem \ref{excess_thm} for some triples $(a,b,c)$, but in order to deal with the remaining cases, we need a stronger upper bound for the number of lattice points in a rational simplex that is also due to Yau and Zhang.  Continuing with the notation above, Yau and Zhang note that $Q_{(\alpha,\beta,\gamma)} = P_{(\alpha(1+\eta),\beta(1+\eta),\gamma(1+\eta))}$.
\begin{thm}\cite[Theorem 1.1]{YZ}\label{strongYZ}
Let $\alpha\ge \beta\ge \gamma \ge 1$ be real numbers.  Then
\[
P_{(\alpha(1+\eta),\beta(1+\eta),\gamma(1+\eta))} \le \frac{(\alpha(1+\eta)-1)(\beta(1+\eta)-1)(\gamma(1+\eta)-1)}{6}.
\]
\end{thm}
There is a sharper upper bound due to Xu and Yau \cite{XY}, but 
it is more challenging to work with we do not need it for our application. 
Using 
Theorem \ref{strongYZ} we can complete the proof of Theorem \ref{excess_thm}.
\begin{proof}[Proof of Theorem \ref{excess_thm}]
Throughout this argument 
let $\alpha,\beta$, and $\gamma$ be defined as in Lemma \ref{simplex_lem}.  Simplifying the right-hand side of Theorem \ref{strongYZ} gives
\[
 \frac{(\alpha(1+\eta)-1)(\beta(1+\eta)-1)(\gamma(1+\eta)-1)}{6} 
 =
 \frac{abc - (ab+ac+bc)+(a+b+c)-1}{6}.
\]
Therefore, we need only show that if $4\le a < b< c$ are pairwise relatively prime positive integers and $(a,b,c) \neq (4,5,7)$, we have 
\[
\frac{abc - (ab+ac+bc)+(a+b+c)-1}{6} \le \frac{abc}{2}- \frac{3}{4}(ab+ac+bc)+ \frac{15}{4};
\]
or equivalently, that
\begin{equation}\label{bound_eq}
F(a,b,c) := \frac{abc}{3} - \frac{7}{12}(ab+ac+bc)-\frac{1}{6}(a+b+c) + \frac{47}{12} \ge 0.
\end{equation}

Differentiating with respect to $c$ shows that for $a\ge 4$ and $b \ge a+1$, $F(a,b,c)$ is an increasing function of $c$.  Therefore, its minimum value occurs when $c = b+1$.  Substituting $c = b+1$ and differentiating with respect to $b$ shows that when $a\ge 4$, $F(a,b,b+1)$ is an increasing function of $b$. 
Therefore, $F(a,b,b+1)$ is minimal when $b=a+1$.  We have $F(a,a+1,a+2) = \frac{1}{3} a^3- \frac{3}{4} a^2 -\frac{10}{3} a + \frac{9}{4}$. Differentiating with respect to $a$, we see that $F(a,a+1,a+2)$ is increasing in $a$ when $a \ge 4$.  We conclude that whenever $4 \le a < b< c$ are pairwise relatively prime positive integers, $F(a,a+1,a+2) \le F(a,b,c)$. Direct computation shows that $F(5,6,7) \ge 0$, so $F(a,b,c) \ge 0$ whenever $a \ge 5$.

\medskip
Finally, suppose $a=4$ and $b=5$. Note that $F(4,5,7)<0$, while $F(4,5,c) > 0$ for $c \ge 9$. Differentiating with respect to $c$ yields $F(4,b,c) \le F(4,b,b+2)$; while differentiating with respect to $b$ yields $F(4,7,9) \le F(4,b,b+2)$.  Direct computation shows that $F(4,7,9) \ge 0$, completing the proof.
\end{proof}

\medskip
The same circle of ideas yields a lower bound on the genus $g=g({\rm S}({\bf r}))$ of the generic semigroup ${\rm S}({\bf r})$ adapted to ${\bf r} = (ab,ac,bc)$, where $a,b,c$ are pairwise relatively prime positive integers.  Here
${\rm cod}(M^3_{d,g;{\rm S},{\bf r}},M^3_d) = ab+ac+bc-7$ whenever $d$ is sufficiently large. Our lower bound on $g({\rm S}({\bf r}))$ yields ${\rm cod}(M^3_{d,g;{\rm S}({\bf r}),{\bf r}},M^3_d) < g({\rm S}({\bf r}))$, and therefore that $M^3_{d,g;{\rm S}({\bf r}),{\bf r}}$ is excess-dimensional, in many cases.

\medskip
To begin, note that ${\rm S}({\bf r})$ clearly contains ${\rm S} = \langle ab,ac,bc\rangle$.  As $abc$ is the smallest nonzero element with more than one factorization in ${\rm S}$, we conclude that ${\rm S}({\bf r}) \cap [1,abc-1] = {\rm S} \cap [1,abc-1]$.  As a result, 
we have
$g({\rm S}({\bf r})) \ge \#\{x \in \mb{N} \setminus {\rm S}\colon x < abc\}$.
On the other hand, the fact that
\[
\#\{x \in {\rm S}\colon x < abc\} + \#\{x \in {\rm \mb{N} \setminus {\rm S}}\colon x < abc\} = abc
\]
implies that to bound $\#\{x \in \mb{N} \setminus {\rm S}\colon x < abc\}$ from below it suffices to bound $\#\{x \in {\rm S}\colon x < abc\}$ from above. More precisely, the corresponding Severi variety will be excess-dimensional whenever
\begin{equation}\label{rhobound2}
\#\{x \in {\rm S}\colon x < abc\}< abc-ab-ac-bc+7.
\end{equation}

As every nonzero element in ${\rm S}$ strictly less than $abc$ factors uniquely it follows, much as before, that the left-hand side of \eqref{rhobound2} is equal to three less than the number of lattice points in the simplex $\Delta(a,b,c)$ with vertices $(0,0,0)$, $(a,0,0)$, $(0,b,0)$, and $(0,0,c)$. 
The estimates of Yau and Zhang
yield the following result.

\begin{thm}\label{excess_thm2}
Let $2 \le a< b< c$ be pairwise relatively prime integers, let  ${\bf r} := (ab,ac,bc)$, let ${\rm S}({\bf r})$ denote the generic semigroup adapted to ${\bf r}$, and assume $d \ge 2g(\rm{S(\bf{r})})$. 
The Severi variety $M^3_{d,g({\rm S}({\bf r}));{\rm S}({\bf r}),{\bf r}}$ is excess-dimensional whenever $a \geq 4$.
\end{thm}

\begin{proof}
Let $\eta=\frac{1}{a}+\frac{1}{b}+\frac{1}{c}$. Applying Proposition~\ref{weakYZ}, it suffices to show that
$\frac{1}{6}(\eta+1)^3<1-\eta+\frac{10}{abc}$.
Whenever $a \geq 5$ and $(a,b,c) \neq (5,6,7)$, we have $1-\eta- \frac{1}{6}(1+\eta)^3>0$ and the conclusion follows. Similarly, according to Theorem~\ref{strongYZ}, it suffices to show that
\begin{equation}\label{strongYZ_abc}
\frac{1}{6}\bigg(\eta+1-\frac{1}{a}\bigg)\bigg(\eta+1-\frac{1}{b}\bigg)\bigg(\eta+1-\frac{1}{c}\bigg)<1-\eta+\frac{10}{abc}.
\end{equation}
By inspection see that \eqref{strongYZ_abc} holds when $(a,b,c)=(5,6,7)$. Finally, say $a=4$, and let $E(x,y):=\frac{1}{6}(1+x+y)(\frac{5}{4}+x)(\frac{5}{4}+y)-\frac{3}{4}+x+y$.
The inequality \eqref{strongYZ_abc} holds in this case whenever $E(x,y)<0$, where $x=\frac{1}{b}$ and $y=\frac{1}{c}$. Clearly $E(x,y)$ is an increasing function in either of its arguments $x$ or $y$, i.e., $E(x,y)$ is {\it decreasing} as a function of either $b$ or $c$. In particular, since $b$ and $c$ are necessarily distinct odd integers, $E(x,y)$ will be maximal when $c=b+2$ for some $b$.
Here
$E(\frac{1}{b},\frac{1}{b+2})$ is negative if and only if $b>5$; it follows that \eqref{strongYZ_abc} holds whenever $(a,b,c) \neq (4,5,7)$; and by direct inspection \eqref{strongYZ_abc} also holds when $(a,b,c)=(4,5,7)$.
\end{proof}

\subsection{First approximations to generic semigroups from supersymmetric triples}\label{Supersym_approx}

Throughout this section we let ${\bf r} = (ab,ac,bc)$ and let $\rm{S}({\bf r})$ be the generic semigroup adapted to ${\bf r}$.  Clearly, $\rm{S} = \langle ab,ac,bc \rangle \subseteq \rm{S}({\bf r})$ and therefore, $g(\rm{S}({\bf r})) \le g(\rm{S}) = abc-\frac{ab+ac+bc-1}{2}$. We now discuss a first approximation to $\rm{S}({\bf r})$ and improve that upper bound on $g(\rm{S}({\bf r}))$ in some cases.  The three factorizations of the element $abc$ in $\rm{S}$ lead directly to the following result.

\begin{prop}\label{generic_semigroup_supersymmetric}
With $2\le a<b<c$ pairwise relatively prime positive integers as above, the value semigroup $\rm{S}({\bf r})$ of a generic parametrization with ramification profile ${\bf r} = (ab,ac,bc)$ contains  $\langle ab, ac, bc; abc+1, abc+2 \rangle$.
\end{prop}

Note that $\langle ab, ac, bc; abc+1, abc+2 \rangle$ is the first approximating semigroup ${\rm S}^{(1)}$ of Remark~\ref{discrete_dynamical_system}.
Instead of working directly with $\langle ab, ac, bc; abc+1, abc+2 \rangle$, we focus on the simpler semigroup $\rm{S}' = \langle ab,ac,bc,abc+1\rangle$ in the case where $abc+1\not\in \langle ab,ac,bc\rangle$.  We begin with a characterization of when this occurs.

\begin{defn}
Let $\gamma$ be the unique integer such that $1 \le \gamma \le c-1$ and $ \gamma ab \equiv 1 \pmod{c}$; let $\beta$ be the unique integer such that $1 \le \beta \le b-1$ and $ \beta ac \equiv 1 \pmod{b}$, and let $\alpha$ be the unique integer such that $1 \le \alpha \le a-1$ and $ \alpha bc \equiv 1 \pmod{a}$.  
\end{defn}

\begin{prop}
The smallest integer in ${\rm S}$ that is congruent to $1$ modulo $abc$ is $\gamma  ab+ \beta  ac + \alpha  bc$.
\end{prop}

\begin{proof}
Suppose that $n = x  ab+ y  ac + z  bc \equiv 1 \pmod{abc}$;
the Chinese remainder theorem then implies that $x \equiv \gamma \pmod{c}, \ \ y \equiv \beta \pmod{b}, \text{ and } z \equiv \alpha \pmod{a}$. So if $x,y,z \ge 0$, we necessarily have $x \ge \gamma,\ y \ge \beta$, and $z \ge \alpha$.
\end{proof}

As $0 < \gamma ab, \beta ac, \alpha bc < abc$, we see that $\gamma  ab+ \beta  ac + \alpha  bc$ is equal to either $abc+1$ or $2abc +1$.  If it is equal to $abc+1$, then ${\rm S}' = {\rm S}$.  For the remainder of this section we focus on the case where $\gamma  ab+ \beta  ac + \alpha  bc = 2abc+1$. This implies $abc+1 = \gamma  ab + \beta  ac + (\alpha-a)  bc$.
\begin{prop}\label{genus_S'}
Suppose $abc+1 \not\in \rm{S}$.  Then 
\[
g(\rm{S'}) = g(\rm{S}) - (a-\alpha)(b-\beta)(c-\gamma).
\]
In this case, $g(\rm{S}({\bf r})) \le abc- \frac{ab+ac+bc-1}{2} - (a-\alpha)(b-\beta)(c-\gamma)$.
\end{prop}

We delay the proof of this proposition to the end of this section.  The main idea of the proof is to count elements $x\in \mb{N}\setminus \rm{S}$ for which $x-(abc+1)\in \rm{S}$.  In order to do this, we begin by giving an explicit description of $\mb{N}\setminus \rm{S}$.

\begin{prop}\label{unique_ab_ac_bc}
Every integer $n$ may be written uniquely as $n = x  ab + y  ac + z  bc$ where $0 \le x \le c-1$ and $0 \le y \le b-1$.
\end{prop}

\begin{proof}
Every residue class modulo $abc$ has a unique representative of the form
\[
m = x'  ab + y'  ac + z'  bc
\]
where $0 \le x \le c-1,\ 0 \le y \le b-1,\ 0 \le z' \le a-1$. Now suppose $m \equiv n \pmod{abc}$.  Then $n = x'  ab + y'  ac + (z'-ka)  bc$ for some integer $k$. To see that this representation is unique, suppose $n = x  ab + y  ac + z  bc = x'  ab + y'  ac + z'  bc$ where $0 \le x,x' \le c-1$ and $0 \le y,' \le b-1$.  As $n \equiv xab \equiv x'ab \pmod{c}$ and $\gcd(c,ab) = 1$, we have $x \equiv x' \pmod{c}$, which implies $x = x'$.  A similar argument shows that $y=y'$.  It follows immediately that $z = z'$, so we conclude.
\end{proof}

\begin{defn}
We say that $(x,y,z)$ is an \emph{$(ab,ac,bc)$-factorization} of an integer $n$ whenever $n = x  ab + y  ac+z bc$.
\end{defn}

\begin{prop}\label{ab_ac_bc_factorizations}
Suppose $(x,y,z)$ is an $(ab,ac,bc)$-factorization of $n$.  The set of $(ab,ac,bc)$-factorizations of $n$ is given by $(x+k_1 c, y+k_2 b, z + k_3 a)$ where $k_1,k_2, k_3 \in \Z$ satisfy $k_1+k_2 + k_3 = 0$.
\end{prop}

\begin{proof}
If $n = x ab + y  ac + z bc$, then for any integers $k_1, k_2, k_3$ with $k_1 + k_2 + k_3 = 0$ we have $(x+k_1 c) ab + (y+k_2 b) ac + (z+k_3 a) bc = n + (k_1+k_2+k_3) abc = n$. Now suppose $n = x'  ab + y'  ac + z'  bc$. As in the proof of the previous proposition, we have $x \equiv x' \pmod{c}$.  Therefore, $x' = x + k_1 c $ for some $k_2 \in \Z$.  A similar argument shows that $y' = y + k_2 b$ for some $k_2 \in \Z$.  Subtracting we deduce
\[
(x ab + y  ac + z bc) - (x'  ab + y'  ac + z'  bc) = (z-z')  bc + (-k_1a+k_2a) bc
\]
which implies that $z' = z - (k_1+k_2)a$.
\end{proof}
Proposition \ref{ab_ac_bc_factorizations} 
amounts to an explicit {\it minimal presentation} for $\langle ab,ac,bc\rangle$ in the sense of \cite[Ch. 7]{GSR}. 

\begin{prop}\label{not_in_ab_ac_bc}
Suppose $n \in 	\Z$ satisfies $n = x  ab + y  ac + z  bc$ where $0 \le x \le c-1$ and $0 \le y \le b-1$.  Then $n \in {\rm S}$ if and only if $z \ge 0$.
\end{prop}

\begin{proof}
If $z \ge 0$ then $n \in {\rm S}$. So suppose $z < 0$.  By Proposition \ref{ab_ac_bc_factorizations}, every $(ab,ac,bc)$-factorization of $n$ is of the form $(x+k_1c, y+k_2b, z + k_3 a)$, where $k_1+k_2 + k_3 = 0$.  If $k_3 \le 0$, we have $z+k_3 a < 0$; if $k_1 < 0$, we have $x+k_1 c <0$; while if $k_2 < 0$, we have $y+k_2 b< 0$.  We conclude that there is no $(ab,ac,bc)$-factorization of $n$ where each coordinate is positive.  So $n \not\in {\rm S}$.
\end{proof}

We can now give the proof of Proposition \ref{genus_S'}.
\begin{proof}[Proof of Proposition \ref{genus_S'}]
As $2(abc+1) > F({\rm S})$, every element of ${\rm S}'\setminus {\rm S}$ is of the form $1 (abc+1) + x  ab + y  ac + z  bc = (x+\gamma) ab + (y+\beta)  ac + (\alpha-a+z) bc$ with $x,y,z \ge 0$. By Proposition \ref{not_in_ab_ac_bc}, the elements of this form that are not in ${\rm S}$ are precisely those for which $0 \le x < c-\gamma,\ 0 \le y < b-\beta$, and $0 \le z < a-\alpha$.  Therefore, there are $(a-\alpha)(b-\beta)(c-\gamma)$ gaps of ${\rm S}$ that are not gaps of ${\rm S}'$.
\end{proof}

Proposition \ref{genus_S'} shows how the genus decreases when we add the element $abc+1$ to $\rm{S}$.  We end this paper by showing how the Frobenius number changes.  An immediate
consequence of the following result is that unlike ${\rm S}$, ${\rm S}'$ is generally not a symmetric semigroup.
\begin{prop}\label{frob_S'}
We have $F({\rm S}') = M$, where
\[
M:= \max\{(\gamma-1)ab + (b-1)ac -bc, (c-1)ab + (\beta-1)ac -bc,(c-1)ab+(b-1)ac + (\alpha-a-1)bc\}.
\]
\end{prop}

\begin{proof}
By Proposition \ref{unique_ab_ac_bc}, every positive integer $n$ satisfies $n = x  ab + y  ac + z  bc$ where $0 \le x \le c-1$ and $0 \le y \le b-1$, and by Proposition \ref{not_in_ab_ac_bc}, the integers not in ${\rm S}$ are precisely those for which $z < 0$.  We saw in the proof of Proposition \ref{genus_S'} that every element of ${\rm S}' \setminus {\rm S}$ is of the form $abc+1 + s = \gamma  bc +\beta  ac + (\alpha-a)  bc + s$ where $s \in {\rm S}$.  That is, every such element is of the form $(\gamma+x)ab+(\beta+y)ac+(\alpha-a + z) bc$ where $x,y,z\in \mb{N}$. We conclude that $M$ is the largest element $n \in \mb{N} \setminus S$ such that $n-(abc+1)$ is also in $\mb{N} \setminus S$.

\medskip
Now suppose $n = x' ab + y' ac + z' bc$, for some $0 \le x \le c-1$ and $0 \le y \le b-1$.
The observations of the preceding paragraph show that $n \not\in {\rm S}'$ if and only if
\begin{enumerate}
\item $z'<0$ and $x' < \gamma$; or
\item $z'<0$ and $y' < \beta$; or 
\item $z'< \alpha - a$.
\end{enumerate}
The largest integers satisfying each of these three conditions are precisely the elements used to define $M$.
\end{proof}

\noindent{\it Acknowledgements.} We thank Sinai Robins for helpful comments. The second author was supported by NSF grants DMS 1802281 and DMS 2154223.
\medskip

\noindent{\bf Declarations.}
\medskip

\noindent{\bf Data availability.} The authors declare that the data supporting the findings of this paper are available within it.

\noindent{\bf Conflict of interest.} On behalf of all authors, the first author states that there is
no conflict of interest.

\end{document}